\documentclass{article}

\usepackage[UKenglish]{babel}

\usepackage{amssymb}  
\usepackage[matrix,arrow,ps]{xy}
\UsePSspecials{dvips}

\usepackage{amsmath,amscd,amsthm, amssymb}

\usepackage{bbm}

\theoremstyle{plain}
\newtheorem{theorem}{Theorem}[section]
\newtheorem{lemma}[theorem]{Lemma}
\newtheorem{proposition}[theorem]{Proposition}
\newtheorem{corollary}[theorem]{Corollary}
\newtheorem{conjecture}[theorem]{Conjecture}

\theoremstyle{definition}
\newtheorem{definition}[theorem]{Definition}
\newtheorem{example}[theorem]{Example}

\theoremstyle{remark}
\newtheorem{remark}[theorem]{Remark}

\newcommand{\bconj}{\begin{conjecture}}
\newcommand{\econj}{\end{conjecture}}

\newcommand{\bd}{\begin{definition}}
\newcommand{\ed}{\end{definition}}

\newcommand{\bp}{\begin{proposition}}
\newcommand{\ep}{\end{proposition}}

\newcommand{\bcor}{\begin{corollary}}
\newcommand{\ecor}{\end{corollary}}

\newcommand{\bthm}{\begin{theorem}}
\newcommand{\ethm}{\end{theorem}}

\newcommand{\bthma}{\begin{thma}}
\newcommand{\ethma}{\end{thma}}
\newcommand{\bthmb}{\begin{thmb}}
\newcommand{\ethmb}{\end{thmb}}

\newcommand{\bl}{\begin{lemma}}
\newcommand{\el}{\end{lemma}}

\newcommand{\bpf}{\begin{proof}}
\newcommand{\epf}{\end{proof}}

\newcommand{\brm}{\begin{remark}}
\newcommand{\erm}{\end{remark}}

\newcommand{\bc}{\begin{center}}
\newcommand{\ec}{\end{center}}

\newcommand{\bex}{\begin{example}}
\newcommand{\eex}{\end{example}}

\newcommand{\bit}{\begin{itemize}}
\newcommand{\eit}{\end{itemize}}

\begin{document}

\title{Word maps with small image in simple groups}
\author{\scshape{Matthew Levy}\\ \small{Imperial College London}}
\date{}
\maketitle

\begin{abstract}
$\vspace{1mm}$
\bc
\begin{minipage}{0.77\textwidth}
We construct non-power words which have small image in SL$(2,2^{2^n})$ for each $n$. In particular, the corresponding word maps are non-surjective. We also use this to construct word maps whose values are precisely the identity and a single equivalence class of elements of order $17$. 

In the second part we construct words which have image consisting of the identity and a single equivalence class of elements  in Alt($n$) for all $n$ for any equivalence class with support size at most $10$.
\end{minipage}
\ec
\end{abstract}

$\vspace{1mm}$
\section{Introduction}
Let $w$ be a word in the free group of rank $k$. For a group $G$, let $G_w$ denote the set of word values, i.e. $G_w:=\{w(g_1,...,g_k)^{\pm 1}|g_i\in G\}$. There has recently been much interest and progress in the study of the images of word maps over finite groups though the topic has grown from work first begun by P. Hall, see \cite{Segal} for a modern exposition. Given a finite group a natural question to ask is which subsets can be obtained as images of word maps. Clearly, any image of a word map includes the identity and must be closed under the action of the automorphism group, i.e. it must be a union of equivalence classes including the equivalence class of the identiy. Here, an equivalence class is a union of conjugacy classes under the action of the automorphism group. It is not clear at first glance whether or not any such subset can be obtained as the image of a word map, though the answer in general is no.

In \cite{KN} it is shown that for any alternating group Alt$(n)$ with $n\ge 5$ and $n\ne 6$ there exists a word $w$ such that Alt$(n)_w$ consists of the identity and all $3$-cycles. For $n\ne 13$ the words they construct are in two variables, for $n=13$ they need three variables. This result also holds for Sym($n$). They also construct words whose image over Alt($n$) is the identity and all $p$-cycles for any prime $3<p<n$ and $n\ge 5$. Note that the exception of Alt($6$) arises due to the outer automorphism which swaps the conjugacy class of $3$-cycles with the $(3,3)$-cycles. As a result it is possible to construct a word with image consisting of the identiy, the $3$-cycles and the $(3,3)$-cycles over Alt($6$). They also give a simple argument to show that it is impossible to construct a word whose values in Sym(n) are either the identity or a transposition which shows that not any subset closed under the action of the automorphism group can be obtained as the image of a word map. Their second main result concerns special linear groups. With the possible exception of SL$_4(2)$ they show that for every $n,q\ge 2$ there is a word $w$ in two variables such that SL$_n(q)_w$ consists of the identity and the conjugacy class of all transpositions. For $n\ne 3,4$ their result also holds for GL$_n(q)$. This paper aims to paritally answer whether or not there exists a word $w_C$ such that $G_{w_C} = \{e\}\cup C$ for any equivalence class $C$ where $G$ is either Alt($n$) or SL$_n(q)$ for $n$, $q\ge 4$, though the question is open for when $G$ is any non-abelian finite simple group.

A conjecture of Shalev (see \cite{BGG}) says that if $w$ is not a proper power of a non-trivial word then the corresponding word map is surjective on PSL$_2(q)$ for sufficiently large $q$. The first counterexamples to Shalev's conjecture are provided in \cite{JLO}. In both \cite{BGG} and \cite{JLO} the authors make use of the \textit{trace polynomial} of a word, this is introduced later. We also go via the trace polynomial to obtain our first main result.

\begin{theorem}\label{main}
For every SL$_2(q)$ with $q=2^{2^n}$ and $n\ge 2$ there exists a word $w$ in $F_2$ such that SL$_2(q)_w$ contains the identity and four conjugacy classes of elements of order 17 and no other conjugacy class consisting of elements of order $17$.
\end{theorem}

Since there are eight conjugacy classes of elements of order $17$ in SL$_2(q)$ which split into two equivalence classes two immediate corollaries are:

\begin{corollary}\label{aner}
For every SL$_2(q)$ with $q=2^{2^n}$ and $n\ge 2$ there exists a word $w$ in $F_2$ such that the corresponding word map is non-surjective. Moreover, $w$ is a non-power word.
\end{corollary}

\begin{corollary}\label{cor}
For every SL$_2(q)$ with $q=2^{2^n}$ and $n\ge 2$ there exists a word $w$ in $F_2$ such that SL$_2(q)_w$ consists of the identity and a single equivalence class of elements of order 17.
\end{corollary}

The second corollary is obtained by taking the word constructed in Theorem \ref{main} and raising to an appropriate power.\\

We will use the results from \cite{KN} and construct new words to obtain our second main result:

\begin{theorem}\label{mainalt}
Let $n\in\mathbb{N}$ and let $C$ denote any equivalence class in Alt$(n)$ with support size at most $10$. Then there exists a word $w_C$ such that Alt($n)_{w_C}=\{e\}\cup C$. 
\end{theorem}

For an equivalence class $C$ in Alt($n$) its \textit{support} is the subset $\{m\in\mathbb{N}|m^a\ne m\text{ for some }a\in\text{Alt}(n)\}$ where Alt($n$) acts in the normal way on [$1,n$]$\subseteq\mathbb{N}$.

\section{Proof of Theorem \ref{main}}
\subsection{Background}

Let $q=p^n$ where $p$ is any prime and $n\ge 1$, then SL$_2(q)$ has order $q(q-1)(q+1)$ and exponent $e=\frac{1}{d}p(q^2-1)$ where $d=$gcd$(2,q-1)$. The elements of SL$_2(q)$ can be classified according to their Jordan forms. For any matrix $A$, its characteristic polynomial is of the form $x^2-tx+1$ where $t=$tr($A$) is the sum of the eigenvaules. This can have $1$, $2$ or none distinct roots (or eigenvalues) in $\mathbb{F}_q$ and in each case, for $A\ne\pm I_2$, the elements are called unipotent, semisimple (split) and semisimple (non-split) resepectively. The conjugacy classes of semisimple elements are uniquely determined by their  trace.  Note that this is not true in general without knowing the order of an element since for example an element of trace $2$ may be the identity or a unipotent element. The table below lists the different classes of elements and gives some information about them. It shows that there is a deep connection between elements in SL$_2(q)$, in terms of their order and trace, and their conjugacy class. 
\begin{center}
    \begin{tabular}{|l|l|l|l|l|l|}
        \hline
        Type & Eigenvalues &  Order & No. conjugacy classes & Size \\ \hline
                                id    &  $1$     										        & $1$                 & $1$                   	     & $1$                      \\ \hline
                              -id    &  $-1$      										        & $d$                 & $1$   	              & $1$                      \\ \hline
                    unipotent    &  $1$           										        & $p$                 & $d$        	              & $\frac{q^2-1}{d}$ \\ \hline
                    unipotent    &  $-1$    										        & $dp$               & $d$          		     & $\frac{q^2-1}{d}$  \\ \hline
semisimple (non-split)    &  $r$, $r^q$ where $r\in\mathbb{F}_{q^2}$ and $r^{1+q}=1$   	        & divides $q-1$ & $\frac{q-d+1}{2}$  & $q(q+1)$               \\ \hline
         semisimple (split)   &  $r$, $1/r$ where $r\in\mathbb{F}_{q}$\textbackslash$\{0,\pm1\}$  & divides $q+1$ & $\frac{q-d-1}{2}$  & $q(q-1)$  	           \\ \hline
    \end{tabular}
\end{center}
$\vspace{5mm}$

The number of distinct conjugacy classes consisting of elements of order $m$ where $m$ divides $q\pm 1$ is $\phi(m)/2$ where $\phi$ is Euler's phi function. To see this, note that there are $\phi(m)$ elements of order $m$ in a cyclic group of order $q\pm 1$. We divide by two because a semisimple matrix of a given order is determined, up to conjugacy, by its pair of eigenvalues and in particular such matrices are conjugate to their inverse, indeed they have equal trace. The number of equivalence classes (under the action of the automorphism group) consisting of elements of order $m$ is $\phi(m)/2k$ where $k$ is the smallest integer such that $p^k\equiv\pm1$ mod $q$, this is the order of the field automorphism $x\mapsto x^p$ modulo inversion. 

\begin{remark}
There are a few things to note. An element of SL$_2(q)$ has order $3$ if and only if tr$(x)=-1$.  Note also that in SL$_2(q)$ where $q$ is as in Theorem \ref{main} there exists $\phi(17)/2=8$ conjugay classes of elements of order $17$ and under the action of the automorphism group there are two equivalence classes of elements of order $17$ each consisting of four conjugacy classes. 
\end{remark}
The following theorem is from \cite{BGK} which is in turn based on a classical theorem of Fricke and Klein. 

\begin{theorem}\label{tracepoly}
Let $F_{2}=<x,y>$ denote the free group of rank two, $G=$SL$_2(q)$ and let tr($M$) be the trace of a matrix $M$. Then for every element $w\in F_2$ there is a unique polynomial $P_w(s,t,u)\in\mathbb{Z}[s,t,u]$ such that tr$(w(A,B))=P_w($tr$(A),$tr$(B),$tr$(AB))$.
\end{theorem}
We will call the polynomial $P_w$ in the theorem above the trace polynomial of $w$ and will sometimes write tr($w$). This theorem provides us with a powerful tool for studying the images of word maps, we can instead study the image of the corresponding trace polynomial. The following identities for traces of $2\times 2$ matrices $A$ and $B$ of determinant $1$ will be used throughout without mention and can be used to find the trace polynomial for specific words:
\begin{align*}
\text{Tr}(A)&=\text{T}r(A^{-1});\\
\text{Tr}(AB)&=\text{Tr}(BA);\\
\text{Tr}(AB)+\text{Tr}(AB^{-1})&=\text{Tr}(A)\text{Tr}(B).
\end{align*}
Using the identities above it is easy to see that the tr($[x,y])=s^2+t^2+u^2-ust-2$, where tr$(x)=s$, tr$(y)=t$ and tr$(xy)=u$. The following lemma makes use of these identities and will be needed in the next section.
\begin{lemma}\label{trace}
Let $v$ denote a group word and let $w=[[v,x],x]$, another group word. Suppose that tr($x)=s$ and let tr($[v,x])=t$. Then tr$(w)=t^2+ts^2$. In particular, if tr$(x)=1$ we have that tr$(w)=t^2+t$.
\end{lemma}
\begin{proof}
First note that tr$([v,x]x)=$tr$([v,x]x^{-1})+$tr$([v,x])$tr$(x)=s(t+1)$. It then follows that tr$(w)=t^2+s^2+s^2(t+1)^2+ts^2(t+1)=t^2+ts^2$.
\end{proof}

The following is a classical theorem and can be found in \cite{Lang}.

\begin{theorem}(Hilbert's Additive Theorem $90$)\label{Hilbert90}
Let $k$ be a field and $K/k$ a cyclic extension of degree $n$ with Galois group $G$. Let $\sigma$ be a generator for $G$. Let $\beta\in K$. The trace $Tr^{K}_{k}(\beta)=0$ if and only if there exists an element $\alpha\in K$ such that $\beta=\alpha - \alpha^{\sigma}$.
\end{theorem}

Here, Tr$^{K}_{k}(\beta)=\sum_{\sigma}\alpha^{\sigma}$, where the sum is over all the Galois conjugates of $\alpha$.

\subsection{Proof of Theorem \ref{main}}
Let $q$ be as in the statement of the theorem and let $M$ denote a fixed element of order $17$ in SL$_2(q)$ with trace $t\in \mathbb{F}_{2^4}$ such that $t$ is a primitive element, i.e. $t$ generates $\mathbb{F}_{2^4}^{*}$. It is easy to check that such a matrix $M$ exists. Then the matrices $M$, $M^2$, $M^4$ and $M^8$ are representatives for four distint conjugacy classes of elements of order $17$ with traces $t$, $t^2$, $t^4$ and $t^8$ respectively. Under the action of the automorphism group these four conjugacy classes merge to form a single equivalence class which we shall dente by $C_1$. The second equivalence class of elements of order $17$, which we shall denote by $C_2$, contains $M^7$ which has trace $t^3$. We aim to construct a word whose corresponding trace polynomial has $t$ in its image but not $t^3$.  This means that the corresponding word map has image consisting of $C_1$ but not $C_2$.

Consider the word $w(x,y)=[[x^{e_3},y^{e_2}],x^{e_3}]$ where $e$ is the exponent of SL$_2(q)$ and $e_n=e/n$. Write $x_3$ for $x^{e_3}$ and $y_2$ for $y^{e_2}$. Then the trace of $x_3$, if it is not the identity, is $1$ whilst $y_2$ has trace $0$. Now tr($[[x_3,y_2],x_3])=u+u^2$ where $u=$tr$(x_3y_2)^2$. It is not hard to see that $u$ can take any value in $\mathbb{F}_q$ as $x$ and $y$ range over the elements of SL$_2(q)$ and since the square map is surjective in a field of characteristic $2$. Let $f(u)=u^2+u$.

Let $w_m(x,y)$ denote the word $[w,_mx_3]$. The next lemma follows easily from \ref{trace}.

\begin{lemma}
Let $w_m$ be as above. Then tr$(w_m)=f^{m+1}(u)$.
\end{lemma}

The next lemma is an easy induction.
\begin{lemma}\label{power}
For each $i$, $f^{2^i}(u)=u^{2^{2^i}}+u$.
\end{lemma}
Theorem \ref{main} will follow immediately from the next lemma.

\begin{lemma}
For $q=2^{2^{2+k}}$, $f^{m+1}$ has $t$ in its image but not $t^3$ where $m=2^{2+k}-2^2$ . Hence, the image of $w_m$ contains $C_1$ but not $C_2$.
\end{lemma}

\begin{proof}
First let $k=0$, so $q=2^4$ and $m=0$. Then $w_m$ has trace polynomial $f(u)=u^2+u$. Suppose $t$ generates $\mathbb{F}_q^*$, then an easy calculation in $\mathbb{F}_q$ shows that Tr$^{\mathbb{F}_q}_{\mathbb{F}_2}(t)=0$ whilst Tr$^{\mathbb{F}_q}_{\mathbb{F}_2}(t^3)=1$. The first case then follows from Hilbert's Theorem $90$.

For general $k$ let $q$ and $m$ be as stated in the theorem. Then $f^{m+1}=f\circ f^{m}$. From lemma \ref{power} it is immediate that the image of $f^{m}$ over $\mathbb{F}_q$ is precisely $\mathbb{F}_{2^4}$ since $m=\sum_{i=0}^{k-1}2^{2+i}$. The general case then follows from the case above.
\end{proof}
The word $w_m^{e/17}$ where $e$ is the exponent of SL$(2,q)$ completes the proof of Corollary \ref{cor}. since the image of the trace polynomial has size $8$ it is not hard to see that the image of the word map $w_m$ consists of $8$ different conjugacy classes  one of which is the identity and four of which makes up $C_1$. A more detailed inspection reveals that the three remaining conjugacy classes must make up the single equivalence class consisting of elements of order $5$. The word $w_m^5$ suffices to complete the proof of Corollary \ref{cor}.

\section{Proof of Theorem \ref{mainalt}}
The proof of our main result is obtained using MAGMA and we will present the results below but first we will explain the general approach. Let $C$ denote an equivalence class in Alt($n$) and denote by $w_C$ the word such that Alt$($n$)_{w_C}=\{e\}\cup C$. We will use the words constructed in \cite{KN} to build the new words that have the image we want, the words we are going to  construct follow a general form. Let $C_1$ and $C_2$ be two equivalent classes in Alt$(n)$ and suppose that we have words $w_{C_1}$ and $w_{C_2}$ with images $\{e\}\cup C_1$ and $\{e\}\cup C_2$ respectively. Consider now the group word $w_{C_1,C_2,e_k}=[w_{C_1},_kw_{C_2}]^{e_k}:=[...[[w_{C_1},w_{C_2}]^{e(1)},w_{C_2}]^{e(2)},...]^{e(k)}$ for some $k$, $e(i)\in\mathbb{N}$ where we write $e_k$ for the vector $(e(i))$. Suppose we want to find a word with image $\{e\}\cup C$ for some equivalence class $C$ in Alt($n$) where $C$ has support size $m$. The idea to pick $C_1$ and $C_2$ and to fix $k$ `large' enough so that the image of $w_{C_1,C_2,((1))}$ contains $C$. We then choose `appropriate' $e(i)$ to kill any unwanted equivalence classes, whilst still retaining $C$, in the image. Also note that if $C_1$ has support size $m_1$ and $C_2$ has support size $m_2$ then we need only check our word maps in Alt($n$) for $m \le n\le m_1+m_2-1$. We will now present the required words to obtain each equivalence class to complete the proof of Theorem \ref{mainalt}:
\begin{enumerate}
\item $(2,2)$-cycles: Take $C_1=3$-cycles, $C_2=3$-cycles and $e_1=(3)$.

\item $(4,2)$-cycles: Take $C_1=5$-cycles, $C_2=5$-cycles and $e_2=(3.5.7,3.5.7)$.

\item $(3,3)$-cycles (for $n\ne 6$): Take $C_1=5$-cycles, $C_2=5$-cycles and $e_2=(3.5.7,4.5.7)$.

\item $(3,2,2)$-cycles: Take $C_1=7$-cycles, $C_2=5$-cycles and $e_3=(3.5.7,4.5,5)$.

\item $(2,2,2,2)$-cycles: Take $C_1=(3,3)$-cycles, $C_2=(3,3)$-cycles and $e_1=(2.3.5.7)$.

\item $(6,2)$-cycles: Take $C_1=(4,2)$-cycles, $C_2=7$-cycles and $e_5=(4.3,2.5.7,4.3.7,4.9.5,5.7)$.

\item $(5,3)$-cycles: Take $C_1=7$-cycles, $C_2=7$-cycles and $e_6=(2.3,7,4.3.7,4.9.5,4.9.5,4.7)$.

\item $(4,4)$-cycles: Take $C_1=7$-cycles, $C_2=7$-cycles and $e_5=(4.5.7,4.5.7,5.7,4.3.5.7,9.5.7)$.

\item $9$-cycles: Take $C_1=7$-cycles, $C_2=7$-cycles and $e_4=(2.9,4.3.5.7,2.9.7,4.5.7)$.

\item $(5,2,2)$-cycles: Take $C_1=7$-cycles, $C_2=7$-cycles and $e_6=(4.7,1,4.5.7,2.5,4.3.5.7,9.7)$.

\item $(4,3,2)$-cycles: Take $C_1=5$-cycles, $C_2=7$-cycles and $e_7=(8.5.7,3.5.7,4.25,81,4.25,4,5.7)$.

\item $(3,3,3)$-cycles: Take $C_1=7$-cycles, $C_2=7$-cycles and $e_1=(4.3.5.7)$.

\item $(4,2,2,2)$-cycles: Take $C_1=5$-cycles, $C_2=9$-cycles and $e_6=(9.5.7,16.7,5.7,8.9.5.11,8.5.7,9.5.7.11)$.

\item $(8,2)$-cycles: Take $C_1=5$-cycles, $C_2=9$-cycles and $e_7=(1,9,1024.9.5,4.9.5,9.5.7,4.9.5.7.11,9.5.7.11)$.

\item $(3,3,2,2)$-cycles: Take $C_1=5$-cycles, $C_2=(3,3,3)$-cycles and $e_7=(1,5,9,5.8.7,9.25.7,81,5.7.11)$.

\item $(7,3)$-cycles: Take $C_1=5$-cycles, $C_2=(3,3,3)$-cycles and $e_7=(2,2.5,1,1,8.7,7^4,8.5.11)$.

\item $(6,4)$-cycles: Take $C_1=5$-cycles, $C_2=9$-cycles and $e_7=(1,8.3.5,25.7,64.9.7.11,8.5.7,8.81.7.11,5.7.11)$.

\item $(5,5)$-cycles: Take $C_1=5$-cycles, $C_2=9$-cycles and $e_7=(4,3,9,5.7,1,4.9.5.7.11,8.9.7.11)$.

\end{enumerate}
\begin{remark}
As already stated, MAGMA was used to carry out the computations to check that the words above have the required image. However, some of the cases are easy to check without the use of MAGMA. For example, to get the $(2,2)$-cycles you need only check that there exists two $3$-cycles, $x$ and $y$, whose commutator, $[x,y]$, is a $(2,2)$-cycle. Then since the commutator of two $3$-cycles lies in Alt($5$) and the only elements of order two in Alt($5$) are the $(2,2)$-cycles the result follows easily. Similar arguments can be used to obtain the $(2,2,2,2)$-cycles and the $(3,3,3)$ cycles.
\end{remark}

\begin{remark}
Together with the results from \cite{KN}, the above shows that any equivalence class with support size at most $11$ can be obtained as the image of a word map in any alternating group.
\end{remark}

\begin{remark}
In \cite{KN}, the authors prove that for every $n,q\ge 2$ with the possible exception of SL$_4(2)$ there is a word $w$ in $F_2$ such that SL$_n(q)_w$ consists of the identity and the conjugacy class of all transpositions. Note that SL$_4(2)$ is isomorphic to Alt($8$) and the conjugacy class of transposistions in SL$_4(2)$ corresponds to the conjugacy class of $(2,2,2,2)$-cycles in Alt($8$). Hence, Theorem \ref{mainalt} deals with this exceptional case.
\end{remark}

\bibliographystyle{plain}
\bibliography{refs}

\begin{thebibliography}{1}

\bibitem{BGG}
T.~Bandman, S.~Garion, and F.~Grunewald.
\newblock On the surjectivity of engel words on {PSL$(2,q)$}.
\newblock {\em Groups, Geometry and Dynamics}, to appear.

\bibitem{BGK}
T.~Bandman, F.~Grunewald, and B.~Kunyavskii.
\newblock Geometry and arithmetic of verbal dynamical systems on simple groups.
\newblock {\em ar$\chi$iv}, (0809.0369).

\bibitem{JLO}
S.~Jambor, M.~W. Liebeck, and E.~A. O'Brien.
\newblock Some word maps that are non-surjective on infinitely many finite
  simple groups.
\newblock {\em ar$\chi$iv}, (1205.1952v1).

\bibitem{KN}
M.~Kassabov and N.~Nikolov.
\newblock Words with few values in finite simple groups.
\newblock {\em ar$\chi$iv}, (1112.5484v1).

\bibitem{Lang}
S.~Lang.
\newblock {\em Algebra}.
\newblock Graduate Texts in Mathematics. Springer, revised third edition
  edition, 2005.

\bibitem{Segal}
D.~Segal.
\newblock {\em Words: notes on verbal width in groups}.
\newblock Cambridge University Press, 2009.

\end{thebibliography}

\end{document}